\documentclass[11pt, oneside]{amsart}
\usepackage{mathtools}
\usepackage{wrapfig}
\usepackage{tikz}
\usepackage{comment}
\usepackage{mathrsfs}
\usepackage{tabularx, hyperref}
\usepackage{amssymb} \usepackage{amsfonts} \usepackage{amsmath}
\usepackage{amsthm} \usepackage{epsfig, subfig}
\usepackage{ amscd, amsxtra, latexsym}
\usepackage{epsfig,  graphicx, psfrag}
\usepackage[all]{xy}
\usepackage{caption}
\usepackage{enumerate}
\usepackage{color}

\DeclareMathOperator{\out}{Out}
\DeclareMathOperator{\supp}{Supp}

\DeclareMathOperator{\mcg}{MCG}

\DeclareMathOperator{\vol}{Vol}

\DeclareMathOperator{\LV}{FV}

\newcommand{\R} {\ensuremath {\mathbb{R}}}

\newcommand{\Z} {\ensuremath {\mathbb{Z}}}

\newcommand{\fil}{\ensuremath{\textrm{fill}}}

\renewcommand{\phi}{\varphi}

\makeatletter
\newtheorem*{var@theorem}{\var@title}
\newcommand{\newvartheorem}[2]{%
	\newenvironment{var#1}[1]{%
		\def\var@title{#2 \ref{##1}}%
		\begin{var@theorem}}%
			{\end{var@theorem}}}
\makeatother

\newvartheorem{thm}{Theorem}
\newvartheorem{prop}{Proposition}

\newtheorem{lemma}{Lemma}[section]
\newtheorem{thm}[lemma]{Theorem}
\newtheorem{prop}[lemma]{Proposition}
\newtheorem{cor}[lemma]{Corollary}

\newtheorem*{prop*}{Proposition}
\newtheorem{prop_intro}{Proposition}

\newtheorem{quest_intro}[prop_intro]{Question}
\newtheorem{thm_intro}[prop_intro]{Theorem}

\newtheorem{cor_intro}[prop_intro]{Corollary}

\theoremstyle{definition}
\newtheorem{defn_intro}[prop_intro]{Definition}

\title[]{Length functions on mapping class groups and simplicial volumes of mapping tori}

\author{Federica Bertolotti}
\address{Scuola Normale Superiore, Pisa, Italy}
\email[F. Bertolotti]{federica.bertolotti@sns.it}

\author{Roberto Frigerio}
\address{Dipartimento di Matematica, Universit\`a di Pisa, Italy}
\email[R. Frigerio]{roberto.frigerio@unipi.it}



\begin{document}
\begin{abstract}
	Let $M$ be a closed orientable manifold. We introduce two numerical invariants, called \emph{filling volumes}, on the mapping class group $\mcg(M)$ of $M$, which are defined in terms of filling norms on the space of singular boundaries on $M$, both with real and with integral coefficients.
	We show that filling volumes are length functions on $\mcg(M)$, we  prove that the real filling volume of a mapping class $f$ is equal to the simplicial volume of the corresponding mapping torus $E_f$, while the integral filling volume of $f$ is not smaller than the stable integral simplicial volume of $E_f$.

	We discuss several vanishing and non-vanishing results for the filling volumes.
	As applications, we show that the hyperbolic volume of $3$-dimensional mapping tori is not subadditive with respect to their monodromy, and  that the real and the integral filling norms on integral boundaries are often non-biLipschitz equivalent.
\end{abstract}

\maketitle

\section{Introduction}
This paper is devoted to the study of two invariants on the mapping class group of a closed orientable manifold. Our invariants (which are in fact the real and the integral version of a unique invariant) are defined in terms of filling norms on the space of singular boundaries, and they turn out to be closely related to the simplicial volume of mapping tori.

Let $R=\R$ or $\Z$, and
let $M$ be a closed orientable $n$-dimensional manifold. Let $C_*(M,R)$ denote the complex of singular chains on $M$ with coefficients in $R$, and
for every $i\in\mathbb{N}$ denote by $Z_i(M,R)\subseteq C_i(M,R)$ (resp.~$B_i(M,R)\subseteq C_i(M,R)$) the subspace of degree-$i$ cycles (resp.~boundaries).
We endow $C_*(M,R)$  with the usual $\ell^1$-norm such that, if $c=\sum_{i\in I} a_i\sigma_i$ is a singular chain written in reduced form, then
\[
	\|c\|_1=\left\| \sum_{i\in I} a_i\sigma_i\right\|=\sum_{i\in I} |a_i|\ .
\]
On the space $B_i(M,R)$ of boundaries there is also defined the \emph{filling norm} $\|\cdot \|_{\fil, R}$ such that, for every $z\in B_i(M,R)$,
\[
	\|z\|_{\fil,R} =\inf \{\|b\|_1\, |\, b\in C_{i+1}(M,R)\, ,\ \partial b=z\}\ .
\]
Recall that $H_n(M,\Z)\cong \Z$ is generated by the \emph{fundamental class} $[M]_\Z$ of $M$, and
that an integral fundamental cycle for $M$ (or $\Z$-fundamental cycle) is just any representative of $[M]_\Z$. We denote by $[M]_\R\in H_n(M,\R)$ the \emph{real} fundamental class of $M$, i.e.~the image of the fundamental class $[M]_\Z$
via the change of coefficient map $H_n(M,\Z)\to H_n(M,\R)$. An $\R$-fundamental cycle (or \emph{real} fundamental cycle) of $M$ is any representative of $[M]_\R$ in $Z_n(M,\R)$.

We denote by $\mcg(M)$ the (positive) mapping class group of $M$, i.e.~the group of homotopy classes of orientation preserving
self-homotopy equivalences of $M$
(in the literature, the mapping class group often denotes the group of isotopy classes of (orientation preserving) self-homeomorphisms;
since there is a natural homomorphism between this last group and the group $\mcg(M)$ as defined above, our invariants are defined also on
isotopy classes of orientation preserving self-homeomorphisms of $M$).

If $f\colon M\to M$ is a map, we denote by $f_*\colon C_*(M,R)\to C_*(M,R)$ the induced map on singular chains.
Observe that $f_*$
is norm non-increasing both with respect to the $\ell^1$-norm,
and (on the subspace of boundaries) with respect to the filling norm.

\begin{defn_intro}\label{L:def}
	Let $f\colon M\to M$ be an orientation preserving self-homotopy equivalence,
	and let $z\in C_n(M,R)$ be an $R$-fundamental cycle for $M$.
	We set
	\[
		\LV_R(f)=\lim_{m\to \infty} \frac{ \|f^m_*(z)-z\|_{\fil,R}}{m}\ ,
	\]
	where $\|\cdot \|_{\fil,R}$ denotes the filling norm on $B_n(M,R)$. (The symbol $\LV$ stands for \emph{filling volume}).
\end{defn_intro}

We will see in Proposition~\ref{initial:prop} that
the invariant $\LV_R(f)$ is well defined, i.e.~the above limit exists, is finite, and  does not depend on the choice of the fundamental cycle $z$.
Moreover, it follows from the very definition that $\LV_\R(f)\leq \LV_\Z(f)$ for every orientation preserving self-homotopy equivalence $f\colon M\to M$.

It turns out that $\LV_R(f)$ only depends on the homotopy class of $f$ (see  Proposition~\ref{homotopy:prop}). With a small abuse,
for every $\varphi\in \mcg(M)$ we will thus denote by $\LV_R(\varphi)$ also the value $\LV_R(f)$, where $f$ is any representative of $\varphi$. In this way, the
invariant $\LV_R$ defines a map
\[
	\LV_R\colon \mcg(M)\to [0,+\infty)\ .
\]

\subsection*{Basic properties of $\LV_R$}
Following~\cite{Ye}, we say that a  \emph{length function} on a group $G$ is a map $l\colon G\to [0,+\infty)$ such that the following conditions hold:
\begin{enumerate}
	\item $l(g^m)=|m|\cdot l(g)$ for every $g\in G$, $m\in\mathbb{Z}$;
	\item $l(ghg^{-1})=l(h)$ for every $g,h\in G$;
	\item $l(gh)\leq l(g)+l(h)$ for every $g,h\in G$ such that $gh=hg$.
\end{enumerate}

\begin{thm_intro}\label{length:thm}
	The map
	\[
		\LV_R\colon \mcg(M)\to \R
	\]
	is a length function.
\end{thm_intro}

The definition above is not the unique (and probably nor the most popular) definition of length function. Indeed, among the properties of length functions, several authors require the
inequality $l(gh)\leq l(g)+l(h)$ to hold for every pair of elements in $G$ (while we require it to hold only for \emph{commuting} elements).
We will see in  Proposition~\ref{non-length:prop} that this further property does not hold in general for our invariants $\LV_\R$ and $\LV_\Z$.

Being a length function, the invariant $\LV_R$ vanishes on every finite-order element of $\mcg(M)$. Using this fact we readily deduce the following:

\begin{cor_intro}\label{Gromov_hyp:cor}
	Let $M$ be a closed connected aspherical and orientable manifold of dimension at least $3$ and  suppose that $\pi_1(M)$ is Gromov hyperbolic.
	Then $\LV_R(f)= 0$ for every orientation preserving self-homotopy equivalence $f \colon M \to  M$.
\end{cor_intro}
Indeed, if $M$ is as in the statement of Corollary~\ref{Gromov_hyp:cor}, then $\out(\pi_1(M))$ is finite~\cite[Theorem~5.4A]{Gromov_hyperbolic}  and $\mcg(M)= \out^+(\pi_1(M ))<\out(\pi_1(M))$
(indeed, it is well known that for every aspherical CW-complex the group of homotopy classes of self-homotopy equivalences is isomorphic to the outer automorphism group of the fundamental group).

Notice that, in dimension not smaller than 3, negatively curved closed connected orientable manifolds satisfy the hypothesis of the corollary above.

\subsection*{The invariant $FV_\R$ and the  simplicial volume of mapping tori}
We  recall the following fundamental:

\begin{defn_intro}
	The \emph{$R$-simplicial volume} of $M$ is
	\[
		\|M\|_R=\inf \{\|z\|_1\, |\, z\ \text{is an}\ R-\text{fundamental cycle for}\ M\}\ .
	\]
\end{defn_intro}

The real simplicial volume $\|M\|_\R$ is just the classical simplicial volume as
defined by Gromov in~\cite{Gromov}, and it is denoted simply by $\|M\|$. Henceforth, when omitting the choice of the coefficients we will always understand that $R=\R$. For example, we will denote simply by
$\LV$, $\|\cdot\|_\fil$, $[M]$ the invariant $\LV_\R$, the norm $\|\cdot \|_{\fil,\R}$, and the real fundamental class $[M]_\R$ of $M$, respectively.

For every homeomorphism $f\colon M\to M$, let us denote by $M\rtimes_{f} S^1$ the mapping torus of $f$,  i.e. the manifold obtained from
$M\times [0,1]$ by identifying $(x,0)$ with $(f(x),1)$ for every $x\in M$. The following result establishes a very neat relationship between $\LV (f)$ and the mapping torus of $f$:

\begin{thm_intro}\label{main:thm}
	Let $M$ be a closed orientable manifold, and let $f\colon M\to M$ be an orientation preserving homeomorphism. Then
	\[
		\LV (f)=\|M\rtimes_{f} S^1\|\ .
	\]
\end{thm_intro}

The inequality $ \|M\rtimes_f S^1\| \leq \LV(f)$ admits a rather simple proof. The proof of  the converse inequality is more involved, and boils down to showing that mapping tori admit efficient fundamental cycles
that are compatible (in a suitable sense) with the fibration (actually, with \emph{any} fixed fibration) of the mapping torus on the circle.

Thanks to Theorem~\ref{main:thm}, the understanding of the invariant $\LV$ can give information on the simplicial volume of mapping tori. For example,
putting together
Theorem~\ref{main:thm} and Proposition~\ref{non-length:prop} we readily deduce the following:

\begin{cor_intro}\label{non-additive:cor}
	Let $\Sigma$ be a closed hyperbolic surface. Then
	there are orientation preserving homeomorphisms $f,g\colon \Sigma\to \Sigma$ such that
	\[
		\|\Sigma\rtimes_{f\circ g} S^1\| >
		\|\Sigma\rtimes_f S^1\|+ \|\Sigma\rtimes_g S^1\|\ .
	\]
\end{cor_intro}

On the other way around, from
Theorem~\ref{main:thm} and the knowledge of the simplicial volume of manifolds that fiber over the circle, one
may deduce many properties of $\LV$.
For example:

\begin{thm_intro}\label{pari}
	Let $n$ be a positive integer, $n\neq 1,3$. Then, there exist a closed orientable $n$-manifold $M$ and a class $\varphi\in \mcg(M)$ such that
	$\LV(\varphi)>0$.
\end{thm_intro}

On the contrary, it is shown in~\cite{Michelletori} that the simplicial volume of any $4$-dimensional closed orientable manifold fibering over the circle vanishes.
This already shows that, if $M$ is a closed orientable $3$-manifold and $\varphi$ is an element of $\mcg(M)$ which may be represented by a homeomorphism, then $\LV(\varphi)=0$.
Moreover, in~\cite{Bertolotti} it is shown that every orientation preserving self-homotopy equivalence of a closed orientable $3$-manifold admits a power homotopic to a homeomorphism.
Building on this two results and on the fact that $\LV_R(\phi^n)= |n|\LV_R(\phi)$ for every $\phi \in \mcg(M)$ and $n \in \Z$, we deduce the following:

\begin{thm_intro}\label{dim3}
	For every closed orientable $3$-dimensional manifold $M$ and every $\varphi\in\mcg(M)$ we have $\LV(\varphi)=0$.
\end{thm_intro}

In contrast, the invariant $\LV_\Z$ does not necessarily vanish on mapping classes of $3$-dimensional manifolds (see Corollary~\ref{ZvsR:cor} below).

Due to the vast literature on the geometry of $3$-manifolds fibering over the circle, in dimension 2 we can collect more information on the invariant $\LV$.
For example, in the case of surfaces
it is possible to completely understand the vanishing or non-vanishing of $\LV$. Indeed, it is well-known that every self-homotopy equivalence of a closed
surface is homotopic to a homeomorphism, and the simplicial volume of mapping tori of hyperbolic surfaces is completely understood
(see e.g.~\cite[Section 2.6]{koji2} for a statement which perfectly fits with our terminology). Hence Theorem~\ref{main:thm} readily implies the following:

\begin{cor_intro}\label{surfaces}
	Let $\varphi\in\mcg(\Sigma_g)$, where $\Sigma_g$ is the closed orientable surface of genus $g\geq 2$. Then $\LV(\varphi)>0$ if and only if $\varphi$ is
	virtually (partially)
	pseudo Anosov, i.e.~if there exist a homeomorphism $f\colon \Sigma_g \to \Sigma_g$ representing $\varphi$ and a subsurface $\Sigma'\subset \Sigma_g$
	such that a power of $f$ leaves $\Sigma'$ invariant and acts as a pseudo Anosov homeomorphism in $\Sigma'$.
\end{cor_intro}

There are plenty of papers relating the hyperbolic volume (hence, via the Proportionality Principle, the simplicial volume) of a 3-manifold fibering over the circle
$\Sigma\rtimes_f S^1$
to other length functions on the mapping class $[f]\in\mcg(\Sigma)$, like the minimal topological entropy, or the translation length with respect to the Teichm\"uller or
the Weil-Petersson distance on Teichm\"uller space (see e.g.~\cite{koji1, koji2, koji3, lackenbypurcell}). In a rather indirect way, these results build a bridge between the invariant $\LV$ and other classical invariants of homeomorphisms of surfaces. For example, Theorem~\ref{main:thm} and the main results of~\cite{koji2, koji3} readily imply the following:

\begin{cor_intro}\label{entropy:cor}
	Let $\Sigma$ be a closed hyperbolic surface, and let $\varphi\in \mcg(\Sigma)$. Then, there exists a constant $C>0$ only depending on the genus of $\Sigma$ such that
	\[
		C^{-1} \|\varphi\|_{WP}\leq \LV (\varphi) \leq C\|\varphi\|_{WP}\
	\]
	and
	\[
		\LV (\varphi) \leq C\|\varphi\|_{T}
	\]
	where $\|\varphi\|_{WP}$ (resp.~$\|\varphi\|_T$) denotes the translation length of the action of $\varphi$ on the Teichm\"uller space of $\Sigma$, with respect to the Weil-Petersson metric
	(resp.~to the Teichm\"uller metric).

	Furthermore, if $\varphi$ is pseudo Anosov, then
	\[
		\LV (\varphi)\leq \frac{ 3\pi |\chi(\Sigma)|}{v_3}  \|\varphi\|_{T}
	\]
	\[
		\LV (\varphi)\leq \frac{3\sqrt{\pi |\chi(\Sigma)|}}{v_3\sqrt{2}}  \|\varphi\|_{WP}
	\]
\end{cor_intro}
Here above, $v_3$ denotes the volume of the (unique up to isometry) regular ideal geodesic simplex in the hyperbolic $3$-space.

It is well known that, for every $\varphi\in \mcg(\Sigma)$, where $\Sigma$ is any hyperbolic surface, the translation length $\|\varphi\|_T$ is equal to the \emph{topological entropy}
$\text{ent}(\varphi)$ of $\varphi$ (see e.g.~\cite{koji2}).
The topological entropy is a well defined topological invariant of self-homeomorphisms of topological spaces, and the topological entropy of a mapping class is the infimum of the topological entropies of its representatives.
In the case of surfaces, Corollary~\ref{entropy:cor} provides an upper bound for the ratio $\LV (\varphi)/\text{ent}(\varphi)$ which only depends on the manifold.
A natural question is whether similar results may hold in higher dimensions. More precisely, we ask here the following:

\begin{quest_intro}
	Let $M$ be a closed orientable $n$-manifold. Does a constant $C>0$ (depending only on $M$) exist such that
	\[
		\LV(\varphi)\leq C\cdot \text{ent} (\varphi)
	\]
	for every $\varphi\in\mcg(M)$?
\end{quest_intro}

When $\Sigma$ is a surface, the quantity $2\pi |\chi(\Sigma)|\text{ent}(\varphi)=\pi \|\Sigma\|\cdot \text{ent}(\varphi)$ is usually called
\emph{normalized entropy} of $\varphi$, and Corollary~\ref{entropy:cor} states that, for surfaces,  the ratio between $\LV (\varphi)$ and the normalized entropy of $\varphi$ is bounded from above by a universal constant. In a previous version of the paper, we asked whether such an estimate (with a constant only depending on the dimension) could hold in higher dimension. However, the authors of~\cite{fiberS}
pointed out to us that this cannot be the case. In fact, should a constant
$C>0$ exist such that
\[
	\LV (\varphi)\leq C\cdot \|M\|\cdot \text{ent} (\varphi)
\]
for every $\varphi\in\mcg(M)$ and every $n$-dimensional orientable closed manifold $M$, we would have
$\|M\rtimes_f S^1\|=0$ for every mapping torus such that $\|M\|=0$. However, in the proof of~\cite[Theorem A and Theorem 2.2]{fiberS} examples are exhibited of
mapping tori $M\rtimes_f S^1$ with $\|M\rtimes_f S^1\|>0$ and $M$ rationally inessential (hence, $\|M\|=0$).

\subsection*{The invariant $\LV_\Z$ and simplicial volumes of mapping tori}
After Gromov's seminal work,
a number of interesting variations of the classical simplicial volume have been studied by several authors. For example, a wide interest has been devoted to the
\emph{stable integral simplicial volume} $\|M\|_\Z^\infty$ of $M$, which is  the infimum of the values of $\|N\|_\Z/d$, as $\|N\|$ varies among the $d$-sheeted coverings of $M$,
and $d$ varies in $\mathbb{N}$. We establish here the following:

\begin{thm_intro}\label{mainZ:thm}
	Let $M$ be a closed orientable manifold, and let $f\colon M\to M$ be a homeomorphism. Then
	\[
		\LV_\Z(f)\geq \|M\rtimes_{f} S^1\|_\Z^{\infty}\ .
	\]
	Moreover, for every $n \geq 2$ there exist a closed $n$-manifold $M$ and a homeomorphism $f \colon M \to M$ such that
	\[
		\LV_\Z(f)>0\, ,\qquad  \|M\rtimes_{f} S^1\|_\Z^{\infty}=0\ .
	\]
\end{thm_intro}

Since $\|M\rtimes_{f} S^1\|_\Z^{\infty}\geq \|M\rtimes_{f} S^1\|=\LV_\R(f)$,
a consequence of the theorem above is the following:

\begin{cor_intro}\label{ZvsR:cor}
	For every $n \geq 2$ there exist an $n$-manifold $M$ and an element $\phi\in\mcg(M)$ such that $\LV_\R(\phi)=0$, while $\LV_\Z(\phi)\neq 0$.
\end{cor_intro}

From Corollary~\ref{ZvsR:cor} it is easily deduced the following:

\begin{cor_intro}\label{filling:cor}
	For every $n \geq 2$ there exists a $n$-manifold $M$ such that
	the restriction to $B_n(M,\Z)$ of the filling norm on $B_n(M,\R)$  is \emph{not} equivalent to
	the integral filling norm. In other words, for every $\varepsilon>0$ there exists an integral boundary $c\in B_n(M,\Z)$ such that
	\[
		\frac{\|c\|_{\fil,\R}}{\|c\|_{\fil,\Z} }<\varepsilon\ .
	\]
\end{cor_intro}

In a different context, the fact that coefficients make a difference when computing filling norms has been recently pointed out in~\cite{Manin}.

In order to prove the second statement of Theorem~\ref{mainZ:thm} above for $n \geq 3$ we need to exploit another variation of the simplicial volume:

\begin{defn_intro}[{\cite{Loh:weight}}]
	The \emph{weightless} $R$-simplicial volume is defined by
	\[
		\|M\|_{(R)}=\min \left\{m\in\mathbb{N}\ \big|\  \sum_{i=1}^m a_i \sigma_i\ \text{is an $R$-fundamental cycle of}\ M\right\}\ .
	\]
\end{defn_intro}
Roughly speaking, weightless simplicial volumes count only the number of simplices in a chain and ignore the (absolute value of the) coefficients.
It is proved in~\cite[Corollary 4.5]{Loh:weight} that the weightless simplicial volume of an even-dimensional hyperbolic manifold $M$ (with coefficients in any principal ideal domain)
is bounded from below by the volume of the manifold, up to a constant depending only on the dimension.
In Section~\ref{integral:sec} we extend this result to hyperbolic manifolds of any dimension by showing the following:

\begin{prop_intro}\label{hyp:easy:prop}
	Let $M$ be an $n$-dimensional closed orientable hyperbolic manifold, and let $R=\R,\Z$ (in fact, $R$ could be any principal ideal domain). Then
	\[
		\|M\|_{(R)}\geq \|M\|=\frac{\vol(M)}{v_n}\ ,
	\]
	where $v_n$ denotes the volume of the (unique up to isometry) regular ideal geodesic simplex in the hyperbolic $n$-space.
\end{prop_intro}


\subsection*{Plan of the paper}
In Section \ref{length:sec} we prove that the filling volume $\LV_R$ is a well defined length function on the mapping class group of  closed oriented manifolds.
We also state and prove basic properties of this map, as the non-subadditivity on hyperbolic surfaces and the fact that $\LV_R$ is constantly zero on manifolds satisfying
the so--called \emph{uniform booundary condition}  in top degree.

Section \ref{real:sec} is mainly devoted to the proof of Theorem \ref{main:thm}, stating that the filling volume of an orientation preserving self-homeomorphism $f$ on an $n$-manifold is equal to the
simplicial volume of the associated mapping torus $M \rtimes_f S^1$. We then deduce Theorem~\ref{pari}  from Theorem~\ref{main:thm}.

In Section \ref{integral:sec} we focus our attention on the integral filling volume $\LV_{\Z}$. We prove that this quantity, for a homeomorphism $f$, is not smaller then the stable integral simplicial volume of the mapping torus relative to $f$.
Finally, with the support of an inequality involving the weightless simplicial volume, we exhibit some examples in which the inequality between the integral filling volume and the stable integral simplicial volume of the
corresponding mapping torus is strict.

\subsection*{Acknowledgement}
The authors thank Bruno Martelli for useful conversations about the complexity and the stable complexity of 3-manifolds fibering over the circle.
We are indebted to Thorben Kastenholz and Jens Reinhold for pointing out to us the existence of manifolds with vanishing simplicial volume admitting
a mapping torus with positive simplicial volume.
We would like to thank Pierre Py and Valentina Disarlo for useful conversations.

\section{The invariant $\LV_R$ as a length function}\label{length:sec}
This section is mainly devoted to the proof of Theorem~\ref{length:thm}, which states that $\LV_R$ is a well defined length function on the mapping class group of closed orientable manifolds.

Throughout the whole section, we fix a closed orientable $n$-manifold $M$.
Let  $f\colon M\to M$ be an orientation preserving self-homotopy equivalence, let $z\in C_n(M,R)$ be an $R$-fundamental cycle for $M$, and recall that
\[
	\LV_R(f)=\lim_{m\to \infty} \frac{ \|f^m_*(z)-z\|_{\fil,R}}{m}\ .
\]

We begin with the following:
\begin{prop}\label{initial:prop}
	The invariant $\LV_R(f)$ is well defined, i.e.~the above limit exists, is finite, and  does not depend on the choice of the fundamental cycle $z$. Moreover,
	\[
		\LV_R(f)=\inf\left\{ \frac{ \|f^m_*(z)-z\|_{\fil,R}}{m}\ \big|\ z\ \text{fundamental cycle for}\ M,\, m\in\mathbb{Z}\right\}\ .
	\]
\end{prop}
\begin{proof}
	Let $z$ be a fixed fundamental cycle for $M$, and let us set
	$a_m=\|f^m_*(z)-z\|_{\fil,R}$. Then, for every $m,m'\in\mathbb{N}$, we have
	\begin{align*}
		a_{m+m'} & =\|f^{m+m'}_*(z)-z\|_{\fil,R} \leq \|f^{m+m'}_*(z)-f^m_*(z)\|_{\fil,R} + \|f^{m}_*(z)-z\|_{\fil,R} \\
		         & = \|f^{m}_*(f^{m'}_*(z)-z)\|_{\fil,R} + \|f^{m}_*(z)-z\|_{\fil,R}                                  \\
		         & \leq \|f^{m'}_*(z)-z\|_{\fil,R} + \|f^{m}_*(z)-z\|_{\fil,R}=a_{m'}+a_{m}.
	\end{align*}
	In other words, the sequence $a_m$ is subadditive, hence, by Fekete's Lemma, the limit $\lim_{m\to +\infty} a_m/m$ exists
	(and is equal to $\inf_{m\in\mathbb{N}} a_m/m$).

	Let now $z,z'$ be $R$-fundamental cycles for $M$. Then $z-z'=\partial c$ for some $c\in C_{n+1}(M,R)$. Hence, for every $m\in\mathbb{N}$,
	$f^m_*(z)-f^m_*(z')=\partial f^m_*(c)$, and
	\[
		(f^m_*(z)-z)-(f^m_*(z')-z')=\partial (f^m_*(c)-c)\ .
	\]
	Therefore,
	\[
		\big| \|f^m_*(z)-z\|_{\fil,R}- \|f^m_*(z')-z'\|_{\fil,R} \big|\leq \|f^m_*(c)-c\|_1\leq 2\|c\|_1\ .
	\]
	As a consequence we get
	\[
		\left| \frac{\|f^m_*(z)-z\|_{\fil,R}}{m}- \frac{\|f^m_*(z')-z'\|_{\fil,R}}{m} \right|\leq  2\frac{\|c\|_1}{m}\ ,
	\]
	hence
	\[
		\lim_{m\to \infty} \frac{\|f^m_*(z)-z\|_{\fil,R}}{m}=\lim_{m\to \infty}\frac{\|f^m_*(z')-z'\|_{\fil,R}}{m}\ .
	\]
	The last statement of the proposition readily follows from Fekete's Lemma.

\end{proof}

The following proposition shows the invariant $\LV_R$ is well defined on mapping classes.

\begin{prop}\label{homotopy:prop}
	Let $f,g\colon M\to M$ be homotopic orientation preserving self-homotopy equivalences. Then
	\[
		\LV_R(f)=\LV_R(g)\ .
	\]
\end{prop}
\begin{proof}
	Let $T_*\colon C_*(M,R)\to C_{*+1}(M,R)$ be the chain homotopy associated to a homotopy between any two homotopic self-homotopy equivalences of $M$. It is well known
	that the operator norm of $T_i$ (with respect to $\ell^1$-norms) is not bigger than $i+1$. Let us fix a fundamental cycle $z$ of $M$ and let $m\in\mathbb{N}$. Since $f$ is homotopic to $g$, also $f^m$ is homotopic to $g^m$, hence
	there is a chain homotopy $T_*\colon C_*(M,R)\to C_{*+1}(M,R)$ such that
	\begin{align*}
		\|(f^m_*(z)-z)-(g^m_*(z)-z)\|_{\fil,R} & =\|f^m_*(z)-g^m_*(z)\|_{\fil,R} \\ &=\|\partial T_n(z)\|_{\fil,R}\leq \|T_n(z)\|_1\leq (n+1)\|z\|_1\ ,
	\end{align*}
	which implies
	\[
		\left| \frac{\|f^m_*(z)-z\|_{\fil,R}}{m}- \frac{\|g^m_*(z)-z\|_{\fil,R}}{m} \right|\leq  \frac{(n+1)\|z\|_1}{m}\ ,
	\]
	whence the conclusion.
\end{proof}

The following proposition concludes the proof of Theorem~\ref{length:thm}.

\begin{prop}
	Let $\varphi,\psi\in \mcg(M)$. Then:
	\begin{enumerate}
		\item $\LV_R(\varphi^m)=|m|\cdot \LV_R(\varphi)$ for every  $m\in\mathbb{Z}$;
		\item $\LV_R(\varphi\psi\varphi^{-1})=\LV_R(\psi)$;
		\item if $\varphi\psi=\psi\varphi$, then $\LV_R(\varphi\psi)\leq \LV_R(\varphi)+\LV_R(\psi)$.
	\end{enumerate}
\end{prop}
\begin{proof}
	Let $f,g\colon M\to M$ be representatives of $\varphi,\psi$, respectively, and let $z$ be an $R$-fundamental cycle for $M$.

	(1): The fact that $\LV_R(f^m)=m \cdot \LV_R(f)$ for every $m\in\mathbb{N}$ is a direct consequence of the definition of $\LV_R$, hence we only need to prove
	that, if $h$ is a homotopy inverse of $f$, then  $\LV_R(h)=\LV_R(f)$. Of course, it is sufficient to show that $\LV_R(h)\leq \LV_R(f)$.

	Let $m\in\mathbb{N}$ be fixed, and let $T_*\colon C_*(M,R)\to C_{*+1}(M,R)$ be the chain homotopy induced by a homotopy between $h^m\circ f^m$ and the identity of $M$.
	As observed in the proof of the previous proposition, in degree $n$ the operator norm of $T_*$ is not bigger than $n+1$. Therefore,
	\begin{align*}
		\|h_*^{m}(z)-z\|_{\fil,R} & =\|h_*^{m}(z)-h^m_*(f^m_*(z))-(z-h^m_*(f^m_*(z)))\|_{\fil,R}                \\
		                          & \leq \|h_*^{m}(z)-h^m_*(f^m_*(z))\|_{\fil,R}+\|z-h^m_*(f^m_*(z))\|_{\fil,R} \\
		                          & \leq \|z-f^m_*(z)\|_{\fil,R}+\|\partial T_n(z)\|_{\fil,R}                   \\
		                          & \leq \|f^m_*(z)-z\|_{\fil,R}+\|T_n(z)\|_{1}                                 \\ &\leq \|f^m_*(z)-z\|_{\fil,R}+(n+1)\|z\|_{1}\ .
	\end{align*}
	By dividing this inequality by $m$ and taking the limit as $m\to \infty$ we get $\LV_R(h)\leq \LV_R(f)$, as desired.

	(2): Let $h$ be a homotopy inverse of $f$, set $z'=h_*(z)$, and fix $m\in\mathbb{N}$. Observe that $(\varphi\psi\varphi^{-1})^m=\varphi\psi^m\varphi^{-1}$ is represented by the map
	$f\circ g^m\circ h$, and let $T_*\colon C_*(M,R)\to C_{*+1} (M,R)$ be a chain homotopy
	(of norm at most $n+1$ in degree $n$)
	between $f_*\circ h_*$ and the identity of $C_*(M,R)$.
	We have
	\begin{align*}
		\|f_*(g_*^m(h_*(z)))-z\|_{\fil,R} & = \|f_*(g_*^m(h_*(z)))-f_*(h_*(z))+f_*(h_*(z))-z\|_{\fil,R}          \\
		                                  & \leq \|f_*(g_*^m(z'))-f_*(z')\|_{\fil,R}+\|f_*(h_*(z))-z)\|_{\fil,R} \\
		                                  & \leq \|f_*(g_*^m(z')-z')\|_{\fil,R}+ \|\partial T_n(z)\|_{\fil,R}    \\
		                                  & \leq \|g_*^m(z')-z'\|_{\fil,R}+ \|T_n(z)\|_1                         \\
		                                  & \leq \|g_*^m(z')-z'\|_{\fil,R}+ (n+1)\|z\|_1\ ,
	\end{align*}
	hence
	\begin{align*}
		\LV_R(\varphi\psi\varphi^{-1}) & =\lim_{m\to +\infty} \frac{\|(fg^m h)_*(z)-z\|_{\fil,R} }{m} \\ &\leq \lim_{m\to +\infty} \frac{\|g_*^m(z')-z'\|_{\fil,R} }{m}+ \lim_{m\to +\infty} \frac{(n+1)\|z\|_1}{m}\\ &=\LV_R(g)+0=\LV_R(\psi)\ .
	\end{align*}
	The very same argument may be exploited to show that
	\[\LV_R(\psi) = \LV_R(\varphi^{-1}(\varphi\psi\varphi^{-1})\varphi)\leq \LV_R(\varphi\psi \varphi^{-1})\ ,\]
	whence the conclusion.

	(3): Since $\varphi\psi=\psi\varphi$, for every $m\in\mathbb{N}$ the map  $(fg)^m$ is homotopic to $f^mg^m$. If $T_*\colon C_*(M,R)\to C_{*+1}(M,R)$ is a chain homotopy
	between $(fg)_*^m$ and $f^m_* g^m_*$, then we have
	\begin{align*}
		\|(fg)_*^m(z)-z\|_{\fil,R} & =\|(fg)_*^m(z)-f^m_*(g^m_*(z))+f^m_*(g^m_*(z))-z\|_{\fil,R}                          \\
		                           & \leq \|(fg)_*^m(z)-f^m_*(g^m_*(z))\|_{\fil,R}+\|f^m_*(g^m_*(z))-z\|_{\fil,R}         \\
		                           & \leq \|\partial T_{n}(z)\|_{\fil,R}+\|f^m_*(g^m_*(z))-f^m_*(z)+f^m_*(z)-z\|_{\fil,R} \\
		                           & \leq \| T_n (z)\|_1+ \|f^m_*(g^m_*(z)-z)\|_{\fil,R}+\|f^m_*(z)-z\|_{\fil,R}          \\
		                           & \leq (n+1)\|z\|_1+\|g^m_*(z)-z\|_{\fil,R}+\|f^m_*(z)-z\|_{\fil,R}\ .
	\end{align*}
	By dividing by $m$ both sides of this inequality and taking the limit as $m\to+\infty$ we get $\LV_R(fg)\leq \LV_R(f)+\LV_R(g)$, as desired.
\end{proof}

One may wonder whether the inequality $\LV_R(fg)\leq \LV_R(f)+\LV_R(g)$ could hold for any pair of orientation preserving self-homotopy equivalences of $M$. However, this is not the case:

\begin{prop}\label{non-length:prop}
	Let $R=\R$ or $R=\Z$, and let $\Sigma$ be a hyperbolic surface. Then, there exist elements $\varphi,\psi\in\mcg (\Sigma)$ such that
	\[
		\LV_R(\varphi\psi)>\LV_R(\varphi)+\LV_R(\psi)\ .
	\]
\end{prop}
\begin{proof}
	Suppose by contradiction the inequality $\LV_R(\varphi\psi)\leq\LV_R(\varphi)+\LV_R(\psi)$ holds for every $\varphi,\psi \in \mcg (\Sigma)$. Then the function $\LV_R$ would be a length function
	according to the definition given, e.g., in~\cite{poly}. However, it is shown in~\cite{poly} that such a function would then factor through the abelianization
	$H_1(\mcg(\Sigma),\mathbb{Z})$ of $\mcg(\Sigma)$. But $H_1(\mcg(\Sigma),\mathbb{Z})$ is finite if the genus of $\Sigma$ is equal to $2$, or zero otherwise, hence
	we would have $\LV_R(\varphi)=0$ for every $\varphi\in\mcg(\Sigma)$. Since the mapping torus of a pseudo-Anosov homeomorphism $f$ of $\Sigma$ is hyperbolic
	(hence, it has positive simplicial volume) this would contradict the fact that
	\[
		0<\|M\rtimes_f S^1\|= \LV_\R([f])\leq \LV_\Z ([f])
	\]
	by Theorem~\ref{main:thm} below. 
\end{proof}

We conclude the section with an easy remark on the behaviour of $\LV=\LV_\R$ on manifolds which satisfy the \emph{uniform boundary condition} property. Following~\cite{Matsu-Mor},
we say that a topological space $X$ satisfies property $n$-UBC (over $R$) if the norms $\|\cdot \|_1$ and $\|\cdot\|_{\fil,R}$ on the space of degree-$n$ boundaries are Lipschitz
equivalent, i.e.~if
there exists $K> 0$ such that, for every boundary $b\in B_{n}(X,R)$, there exists
a chain $c\in C_{n+1}(X,R)$ such that $\partial c=b$ and $\|c\|_1\leq K\cdot \|b\|_1$.

\begin{prop}
	Suppose that $M$ satisfies $n$-UBC. Then $\LV (\varphi)=0$ for every $\varphi\in\mcg(M)$.
\end{prop}
\begin{proof}
	If $f\colon M\to M$ is an orientation preserving self-homotopy equivalence, and $K> 0$ is as in the above definition of condition $n$-UBC, then
	\begin{align*}
		\LV(f) & =\lim_{m\to \infty} \frac{ \|f^m_*(z)-z\|_{\fil}}{m}\leq \lim_{m\to \infty} \frac{ K\cdot \|f^m_*(z)-z\|_{1}}{m} \\ &\leq
		\lim_{m\to \infty} \frac{ K\cdot \left(\|f^m_*(z)\|_1+\|z\|_{1}\right)}{m}\leq \lim_{m\to \infty} \frac{ 2K\cdot \|z\|_{1}}{m}=0\ .
	\end{align*}
\end{proof}

It  was shown in~\cite{Matsu-Mor} that, if $\pi_1(M)$ is amenable, then $M$ satisfies $n$-UBC over $\R$ for every $n\geq 1$. We then get the following:

\begin{cor}\label{amenable:cor}
	If $\pi_1(M)$ is amenable, then $\LV (\varphi)=0$ for every $\varphi\in\mcg(M)$.
\end{cor}

Corollary~\ref{amenable:cor} may be easily deduced also from Theorem~\ref{main:thm}, together with the  fact that
the simplicial volume of any fiber bundle with an amenable fiber of positive dimension vanishes~\cite{Gromov,FriMoM,Loh-Sauer-ens}.

Generalizations and variations of the  uniform boundary condition (which involve also integral coefficients) are introduced and studied
in~\cite{Fauser-Loh}. We remark here that in Proposition~\ref{Anosov:prop} we will show that $\LV_\Z$ does \emph{not} vanish on many mapping classes
on the $2$-dimensional torus. Since the torus has an amenable fundamental group, this shows that Corollary~\ref{amenable:cor} does not hold if one replaces
$\LV_\R$ with $\LV_\Z$.

\section{Simplicial volume of mapping tori}\label{real:sec}
Throughout the whole section, we fix  a closed orientable $n$-dimensional manifold $M$.
In this section we prove  Theorem~\ref{main:thm}, which states that $\LV(f)=\LV_\R(f)$ is equal to the simplicial volume of the mapping torus of $f$, for every
orientation preserving self-homeomorphism $f$ of $M$.
We begin with the following lemma, which holds both for real and for integral coefficients:

\begin{lemma}\label{easy:estimate}
	Let $f\colon M\to M$ be a homeomorphism, and let $z$ be an $R$-fundamental cycle for $M$. Then
	\[
		\|M\rtimes_{f} S^1\|_R\leq  (n+1)\|z\|_1+ \|f_*(z)-z\|_{\fil,R}\ .
	\]
\end{lemma}
\begin{proof}
	For $j=0,1$, let $i_j\colon M\to M\times [0,1]$ be the inclusion defined by $i_j(x)=(x,j)$, and let $\pi\colon M\times [0,1]\to M\rtimes_{f} S^1$ be the quotient
	map which identifies $(x,0)$ with $(f(x),1)$ for every $x\in M$. The standard chain homotopy between $(i_0)_*$ and $(i_1)_*$ provides
	a chain  $\overline{z}\in C_{n+1}(M\times [0,1],R)$ such that $\partial \overline{z}=(i_0)_*(z)-(i_1)_*(z)$ and $\|\overline{z}\|_1\leq (n+1)\|z\|_1$. In order to project $\overline{z}$
	onto a fundamental cycle for $M\rtimes_{f} S^1$ we need to suitably replace the summand $(i_1)_*(z)$ with $(i_1\circ f)_* (z)$. To this aim, for any given $\varepsilon>0$
	we can choose a chain $b\in C_{n+1}(M,R)$ such that $\partial b= z - f_*(z)$ and $\|b\|_1\leq  \|z -f_*(z)\|_{\fil,R}+\varepsilon$, and we  set $\widehat{z}=\overline{z}+(i_1)_*(b)\in C_{n+1}(M\times [0,1])$. By construction we have
	\[
		\partial \widehat{z}=(i_0)_*(z)-(i_1\circ f)_* (z)\, ,
	\]
	hence $\pi_*(\widehat{z})$ is an $R$-fundamental cycle for $M\rtimes_{f} S^1$. We thus get
	\begin{align*}
		\|M\rtimes_{f} S^1\|_R & \leq
		\|\widehat{z}\|_1 \leq \|\overline{z}\|_1+\|(i_1)_*(b)\|_1 \\ &\leq (n+1)\|z\|_1+\|z-f_*(z)\|_{\fil,R}+\varepsilon\ .
	\end{align*}
	The conclusion follows thanks to the arbitrariness of $\varepsilon$.
\end{proof}

We now restrict our attention to the case with real coefficients, and  prove the main result of this section:

\begin{varthm}{main:thm}
	Let $M$ be a closed orientable manifold, and let $f\colon M\to M$ be an orientation preserving homeomorphism. Then
	\[
		\LV (f)=\|M\rtimes_{f} S^1\|\ .
	\]
\end{varthm}
\begin{proof}
	We begin by proving the inequality
	\begin{equation}\label{easy:in}
		\|M\rtimes_f S^1\| \leq \LV(f)\ .
	\end{equation}
	Let $z$ be any fundamental cycle for $M$, and let $m\in\mathbb{N}$.
	By applying Lemma~\ref{easy:estimate} to the homeomorphism $f^m$
	we get
	\begin{equation*}
		\|M\rtimes_{f^m} S^1\|\leq (n+1)\|z\|_1+ \|f^m_*(z)-z\|_\fil\ .
	\end{equation*}
	Observe now that $M\rtimes_{f^m} S^1$ is the total space of a degree-$m$ covering of $M\rtimes_f S^1$. Since the simplicial volume is multiplicative with respect to finite coverings, this yields
	\[
		\|M\rtimes_{f} S^1\|=\frac{\|M\rtimes_{f^m} S^1\|}{m}\leq \frac{(n+1)\|z\|_1}{m}+\frac{\|f^m_*(z)-z\|_\fil}{m}\ .
	\]
	By taking the limit in this inequality as $m\to +\infty$, we then get inequality~\eqref{easy:in}.

	In order to conclude, we are now left to show the inequality
	\[\|M \rtimes_f S^1\| \geq \LV(f)\ ,\]
	which requires more work.

	Let
	\[E = M \rtimes_{f} S^{1} = \frac{M \times [0,1]}{(f(x),1)\sim(x,0)}\] be the mapping torus of $f$. For every $m\in\mathbb{N}$,
	we consider the degree-$m$ cyclic cover $\pi_{m}\colon E_{m} \to E$ of $E$ and the infinite cyclic cover $\pi\colon\widetilde{E} \to E$ with total spaces
	\[ E_{m} = M \rtimes_{f^m} S^1 = \frac{ M \times [0,m] } {(f^{m} (x) ,m) \sim(x,0)}\]
	\[\widetilde E = M \times \R. \]
	For every fundamental cycle $z$ of $E$,  and for $m$ big enough, we will construct a relative fundamental cycle $w$ of $(M \times [0,m],M \times \{0,m \})$ such that $w$ projects
	onto a fundamental cycle of $E_{m}$, $ \|w \|_1 \leq m \cdot \|z \|_1$ and  $\partial w = b - f^{m}_* (b)$ for some fundamental cycle $b$ of $M$.

	Let us fix a fundamental cycle of $E$
	\[z = \sum_{i = 1}^{s} x_{i} \sigma_{i}	 \ .\]

	We define the \emph{length} of a singular simplex $\sigma$  in $E$ (or in $\widetilde{E}_m$) as follows:
	we consider a lift $\tilde \sigma$ of $\sigma$ in $\widetilde E=M\times\R$ and project its image on $\R$, thus obtaining an interval $[a_{\tilde\sigma},b_{\tilde\sigma}]$;
	we then set $\operatorname{length} (\sigma)= b_{\tilde\sigma}- a_{\tilde\sigma} $ (this definition does not depend on the choice of the lift
	$\tilde \sigma$). We finally denote by $N$ the maximal length of simplices appearing in $z$, i.e.~we set
	\[N=\max \{\operatorname{length}(\sigma_{i})\, |\, 1 \leq i \leq s \}.\]

	Let now $z_{m}$ be a collection of lifts of $z$ on $E_m$, so that $z_m$ is a fundamental cycle of $E_{m}$ satisfying $(\pi_{m})_{*} (z_{m})= m \cdot z $ and $\|z_{m} \|_1= m \|z \|_1$.
	Observe that every simplex appearing in $z_m$ has length not bigger than $N$.

	Let us fix a natural number $m > 2N + 1$, and consider a continuous map $h\colon [0,m]\to [0,m]$ such that
	$h([0,N])=\{0\}$ and $h([m-N,m])=\{m\}$.
	We extend $h$ by $m$-periodicity to a map (still denoted by $h$) defined on the whole real line. By construction,
	for every $k\in\mathbb{Z}$ the map
	$h$ shrinks the interval $[km-N,km+N]$ onto the single value $km$.
	We denote by $\overline{g}\colon M\times \R\to M\times \R$
	the map defined by $\overline{g}(x,t)=(x,h(t))$, so that
	$\overline{g}(M\times [km-N,km+N])=M\times \{km\}$ for every $k\in\mathbb{Z}$.

	Being equivariant with respect to the action of the  automorphisms of the covering
	$\widetilde{E}\to E_m$, the map $\overline{g}$ induces
	a continuous map $g\colon E_m\to E_m$. Since $\overline{g}$ is equivariantly homotopic to the identity of $\widetilde{E}$,
	the map $g$ is homotopic to the identity of $E_m$, hence the chain $g_*(z_m)$ is still a fundamental cycle for $E_m$.
	Moreover $\|g_*(z_m)\|_1\leq \|z_m\|_1$.


	Let us now denote by $\widetilde{z}$ the chain obtained by lifting $z$ (or $z_m$) to $\widetilde{E}$. Since $z$ is a fundamental cycle for $E$, the chain
	$\widetilde{z}$ is a locally finite (but infinite) fundamental cycle for $\widetilde{E}$ (see e.g.~\cite{Lothesis} for the definition of fundamental cycle
	of a non-compact manifold).
	Let $\widetilde{s}$ be
	a singular simplex appearing in $\widetilde{z}$ (i.e.~$\widetilde{s}$ is
	a lift of a simplex appearing in $z_m$
	under the covering $\widetilde{E}\to E_m$).

	Suppose first that  $\supp (\widetilde{s})\cap (M\times \{km\})\neq \emptyset$ for some $k\in\mathbb{Z}$. Since the length
	of any simplex appearing in $z_m$ is not bigger then $N$, the image of $\widetilde{s}$ is contained
	in $M\times [km-N,km+N]$, hence the image of $\overline{g}_*(\widetilde{s})$ is contained in $M\times \{km\}$.

	On the other hand, if $\supp  (\widetilde{s})\cap (M\times \{km\})= \emptyset$ for every $k\in\mathbb{Z}$, then there exists $k_0\in \mathbb{Z}$
	such that $\supp  (\widetilde{s})\subseteq M\times (k_0m,(k_0+1)m)$.
	Since $h$ maps the open interval $(k_0m,(k_0+1)m)$ onto the closed interval $[k_0m,(k_0+1)m]$, this implies
	that $\supp (\overline{g}_*(\widetilde s))\subseteq M\times [k_0m,(k_0+1)m]$.

	In any case, the image of any singular simplex appearing in
	$\overline{g}_*(\widetilde{z})$ is contained in a subset of $\widetilde{E}$ of the form $M\times [k_0m, (k_0+1)m]$ for some $k_0\in\mathbb{Z}$.
	Also observe that $\overline{g}_*(\widetilde{z})$ obviously coincides with the lift to $\widetilde{E}$ of the fundamental cycle $g_*(z_m)$ of $E_m$.

	Let us write $g_{*} (z_{m})$ as a linear combination of simplices:
	\[g_{*} (z_{m})=\sum_{i = 1}^{m s} y_{i} s_i\ .\]
	The idea is ``to cut'' $g_{*} (z_{m})$ in order to obtain a relative fundamental cycle $w$ of $(M \times [0,m],M \times \{0,m \})$.
	For every $i=1,\dots,m s$, we denote by $\widetilde{s}_i\colon \Delta^{n+1}\to \widetilde{E}=M\times \R$ the unique
	lift of $s_i$ such that the following condition holds: the image of $\widetilde{s}_i$ is disjoint from $M\times (-\infty,0)$, but is not
	disjoint from $M\times [0,m)$. Since every singular simplex appearing in
	$\overline{g}_*(\widetilde{z})$ is contained in a subset of $\widetilde{E}$ of the form $M\times [k_0m, (k_0+1)m]$, we have
	$\supp(\widetilde{s}_i)\subseteq M\times [0,m]$ for every $i=1,\dots,ms$.

	Let us now consider the chain
	\[w = \sum_{i=1}^{m  s} y_i \tilde{s}_i\ .\]
	By construction we have
	\[\|w\|_1=\|g_*(z_m)\|_1\leq \|z_m\|_1=m\|z\|_1\ .\]
	We are going to show that $w$
	is indeed a relative fundamental cycle for $(M \times [0,m] , M \times \{0,m \})$.

	Let $\tau\colon \widetilde{E}\to \widetilde{E}$ be the generator of the group of the automorphisms of the covering $\widetilde{E}\to E_m$, i.e.~let
	$\tau(x,t)=(f^m(x),t+m)$ for every $(x,t)\in M\times\R=\widetilde{E}$. By construction, we have
	\[
		\overline{g}_*(\widetilde{z})=\sum_{j\in\mathbb{Z}} \tau^j_*(w)\ .
	\]
	Moreover, $\supp(w)\subseteq M\times [0,m]$, hence $\supp (\partial  \tau^j_*(w))\subseteq M\times [jm,(j+1)m]$.
	Since $\partial \overline{g}_*(\widetilde{z})=0$, this readily implies that $\supp(\partial w)\subseteq (M\times\{0\})\cup(M\times \{m\})$, i.e.
	$w$ is a relative cycle for the pair $(M\times [0,m],\partial (M\times [0,m]))$.

	More precisely, we have $\partial \tau^j_*(w)=\tau^j_*(b^+)-\tau^j_*(b^-)$, where $b^+$ is supported in $M\times \{m\}$ and $b_j^-$ is supported
	in $M\times \{0\}$.
	Of course we have $\supp(\tau^j_*(b^+))=\tau^j(\supp(b^+))\subseteq M\times \{(j+1)m\}$, and $\supp(\tau^j_*(b^-))=\tau^j(\supp(b^-))\subseteq M\times \{jm\}$.
	From the condition $\partial \widetilde{z}=0$ we then deduce that
	\begin{equation}\label{bordi:eq}
		b^+=\tau_*(b^-)\ .
	\end{equation}

	Let us now prove that $w$ is a fundamental cycle for the pair $(M\times [0,m],\partial (M\times [0,m]))$.
	To this aim, let us pick a point $q\in M\times (0,m)$.
	Being a fundamental cycle for $E_m$, the chain $g_*(z_m)$ represents the positive generator of
	$H_{n+1}({E}_m, {E}_m\setminus \{p_m(q)\})\cong \mathbb{Z}$.
	Using that
	the covering projection $p_m\colon \widetilde{E}\to E_m$ is an orientation preserving local homeomorphism, this implies that
	$\overline{g}_*(\widetilde{z})$ is a representative of the positive generator of $H_{n+1}(\widetilde{E}, \widetilde{E}\setminus \{q\})$. But  the simplices
	appearing in $\overline{g}_*(\widetilde{z})-w$ are all supported outside $M\times (0,m)$, hence
	$[w]=[\overline{g}_*(\widetilde{z})]$ in $H_{n+1}(\widetilde{E}, \widetilde{E}\setminus \{q\})$, and this finally implies that $w$ is a fundamental cycle for $(M\times [0,m], \partial (M\times [0,m])$.

	As a consequence, $b^-$ is a fundamental cycle for $M\times\{0\}$.
	If $k\colon M\times [0,m]\to M$ is the projection onto the first factor and $b=k_*(b^-)$, then $b$ is a fundamental cycle for $M$, and
	Equation~\eqref{bordi:eq} implies that $\partial k_*(w)=f^m_*(b)-b$. We thus get
	\[
		\frac{\|f^m_*(b)-b\|_\fil}{m}\leq \frac{ \|k_*(w)\|_1}{m}\leq \frac{\|w\|_1}{m}\leq \|z\|_1\ .
	\]

	Let us summarize what we have proven so far: for any given fundamental cycle $z$ for $M\rtimes_f S^1$, we have constructed
	a fundamental cycle $b$ for $M$ and a natural number $m\in\mathbb{N}$ such that
	\[
		\frac{\|f^m_*(b)-b\|_\fil}{m}\leq \|z\|_1\ .
	\]
	By the last statement of Proposition~\ref{initial:prop}, this implies that $\LV(f)\leq \|M\rtimes_f S^1\|$, as desired.
\end{proof}

A consequence of the previous result is that our length function is non trivial in every dimension different from $1$ and $3$.

\begin{varthm}{pari}
	Let $n$ be a positive integer, $n\neq 1,3$. Then, there exist a closed orientable $n$-manifold $M$ and a class $\varphi\in \mcg(M)$ such that
	$\LV(\varphi)>0$.
\end{varthm}

\begin{proof}
	Thanks to Theorem \ref{main:thm}, it suffices to find examples of $n$-dimensional mapping tori with positive simplicial volume for every $n \neq 2,\,4$. Such examples are given in \cite[Corollary 1.4]{Michelletori}.
\end{proof}

The examples of mapping tori with non-vanishing simplicial volume described in~\cite[Corollary 1.4]{Michelletori} are obtained by taking the product of $3$-dimensional mapping
tori with arbitrary manifolds with positive simplicial volume.

The $3$-dimensional case  is well known: indeed, for every hyperbolic surface $\Sigma$ and every pseudo Anosov map $f \colon \Sigma \to \Sigma$, the mapping torus $\Sigma \rtimes_f S^1$ is hyperbolic, hence it has positive simplicial volume.

Other interesting examples have been recently discovered in any odd dimension $n\geq 5$.
For $n = 5$, one can consider the mapping torus $V$ constructed in \cite{fujiwara} starting from the finite volume (noncompact) hyperbolic manifold described in \cite{brunoeco}.
The closed $5$-manifold $V$ has positive simplicial volume: indeed, $V$ is aspherical (being nonpositively curved) and $\pi_1(V)$ is relatively hyperbolic, hence $\|V\|>0$ by \cite[Proposition 1.6]{belegradek}.
As an alternative, in order to show that $\|V\|>0$ one can apply \cite[Corollary 1.6]{connell}, since $V$ has nonpositive sectional curvature and there exists a point $p \in V$ so that the sectional curvature along each tangent plane at $p$ is negative.


Finally, we refer the reader to~\cite[Corollary B]{fiberS} for other interesting constructions in odd dimensions $n\geq 7$.

\section{On the invariant $\LV_\Z$}\label{integral:sec}
As mentioned in the introduction, the interpretation of the invariant $\LV_\Z$ seems to be less clear than for its real counterpart. We first prove the following proposition,
which relates $\LV_\Z(f)$ with the stable integral simplicial volume of the mapping torus of $f$.

\begin{prop}\label{stable:prop}
	Let $M$ be a closed orientable manifold, and let $f\colon M\to M$ be a homeomorphism. Then
	\[
		\LV_\Z(f)\geq \|M\rtimes_{f} S^1\|_\Z^{\infty}\ .
	\]
\end{prop}
\begin{proof}
	For every $m\in\mathbb{N}$, the mapping torus $M\rtimes_{f^m} S^1$ is the total space of a degree-$m$ cover of $M\rtimes_{f} S^1$, hence by definition
	of stable integral simplicial volume we have
	\begin{equation*}
		\|M\rtimes_{f} S^1\|_\Z^{\infty}\leq \frac{ \|M\rtimes_{f^m} S^1\|_\Z}{m}\ .
	\end{equation*}
	Let now $z$ be an integral fundamental cycle for $M$. By applying Lemma~\ref{easy:estimate} to the homeomorphism $f^m\colon M\to M$ and dividing by $m$ we get
	\[
		\frac{ \|M\rtimes_{f^m} S^1\|_\Z}{m}\leq  \frac{(n+1)\|z\|_1}{m}+\frac{ \|f^m_*(z)-z\|_{\fil,\Z}}{m}\ ,
	\]
	hence
	\[
		\|M\rtimes_{f} S^1\|_\Z^{\infty}\leq \frac{(n+1)\|z\|_1}{m}+\frac{ \|f^m_*(z)-z\|_{\fil,\Z}}{m}\ .
	\]
	The conclusion follows by taking the limit of this inequality as $m\to +\infty$.
\end{proof}
As mentioned in the introduction, in order to show that the converse inequality does not hold in general we first bound from below the weightless simplicial volume of hyperbolic manifolds:

\begin{varprop}{hyp:easy:prop}
	Let $M$ be an $n$-dimensional closed orientable hyperbolic manifold, and let $R=\R,\Z$ (in fact, $R$ could be any principal ideal domain). Then
	\[
		\|M\|_{(R)}\geq \|M\|=\frac{\vol(M)}{v_n}\ .
	\]
\end{varprop}
\begin{proof}
	Let $\sum_{i=1}^k a_i\sigma_i$ be an $R$-fundamental cycle for $M$ realizing the weightless simplicial volume of $M$, i.e.~assume that $k=\|M\|_{(R)}$.
	The straightening operator on singular simplices (see e.g.~\cite[Chapter 8]{frigerio:book}) induces a chain
	map $S_*\colon C_*(M,R)\to C_*(M,R)$ which is homotopic to the identity, hence the cycle
	$z=\sum_{i=1}^k a_i S_*(\sigma_i)$ is still an $R$-fundamental cycle for $M$.

	Let us now denote by $\supp(S_*(\sigma_i))$ the image of $S_*(\sigma_i)$ in $M$. We observe that  $\bigcup_{i=1}^k \supp(S_*(\sigma_i))=M$. Indeed, otherwise $z$
	would be supported in $M\setminus \{p\}$ for some point $p\in M$, hence its class $[z]\in H_n(M,R)$ would lie in the image of the map
	$H_n(M\setminus \{p\},R)\to H_n(M,R)$ induced by the inclusion. But $H_n(M\setminus \{p\},R)=0$, hence $[z]$ would be the zero class, against the hypothesis that $z$ is an $R$-fundamental cycle for $M$.

	Now the conclusion follows from an easy computation involving the volume of straight simplices:
	by definition, $S_*(\sigma_i)$ is the projection  in $M$ (via a locally isometric map) of a geodesic simplex in the hyperbolic $n$-space,
	and the volume of every geodesic simplex in the hyperbolic $n$-space is bounded above by $v_n$~\cite{HM,Pe}, hence
	\[
		\vol(M)=\vol\left( \bigcup_{i=1}^k \supp(S_*(\sigma_i))\right)\leq \sum_{i=1}^k \vol (\supp(S_*(\sigma_i)))\leq kv_n\ ,
	\]
	which implies
	\[
		\|M\|_{(R)}=k\geq \frac{\vol(M)}{v_n}=\|M\|\ .
	\]
\end{proof}

We are now ready to prove that $\LV_\Z$ can be non-null on mapping classes on which $\LV_\R$ vanishes.
More precisely, we prove the following result, which in turn proves the second statement of Theorem~\ref{mainZ:thm} for $n \geq 3$:
\begin{thm}\label{ZvsR:thm2}
	Let $\Sigma$ be a hyperbolic surface, and $g\colon \Sigma\to\Sigma$ be a pseudo Anosov homeomorphism.
	Let $M=\Sigma\times X \times S^1 $, where $X$ is any closed orientable manifold (possibly a point),
	and set $f\colon M\to M$, $f(x,\alpha,\beta)=(g(x),\alpha,\beta)$. Then
	\[
		\|M\rtimes_f S^1\|^{\infty}_\Z=0\ ,\qquad \LV_\Z(f)> 0\ .
	\]
\end{thm}
\begin{proof}
	For every $m\in\mathbb{N}$, let $N_m=\Sigma\rtimes_{g^m} S^1$ be the mapping torus of $g^m$. Since $g$ is pseudo Anosov, $N_m$  is a closed orientable hyperbolic $3$-manifold.
	Since $f^m$ acts as the identity on the factor $X \times S^1$, it is readily seen that the mapping torus $M\rtimes_{f^m} S^1$ splits as a product
	$M\rtimes_{f^m} S^1=N_m\times (X \times S^1)$. In particular, since  $\|V\times S^1\|_\Z^\infty=0$ for every closed orientable manifold $V$,  we have that
	\[\|M\rtimes_f S^1\|^{\infty}_\Z=\|N_1 \times X \times S^1\|^{\infty}_\Z= 0\ .\]

	Let us now prove that $\LV_\Z(f)>0$. Let us fix $m\in\mathbb{N}$ and an integral fundamental cycle $z$ for $M$. First observe that \cite[Proposition 2.10]{Loh:weight} implies that
	\begin{align*}
		\|M\rtimes_{f^m} S^1\|_\Z & =\|N_m \times (X\times S^1)\|_\Z                     \\
		                          & \geq \|N_m \times (X\times S^1)\|_{(\Z)}             \\
		                          & \geq \max \{\|N_m\|_{(\Z)}, \|X\times S^1\|_{(\Z)}\} \\
		                          & \geq \|N_m\|_{(\Z)}\ .
	\end{align*}
	By Proposition~\ref{hyp:easy:prop} we now have $\|N_m\|_{(\Z)}\geq \|N_m\|=m\cdot \|N_1\|$, where the last equality is due to the fact that $N_m$ is the total space of a degree-$m$
	cover of $N_1$. We thus get
	\[
		\|M\rtimes_{f^m} S^1\|_\Z\geq m \cdot \|N_1\|\ .
	\]
	Therefore, applying Lemma~\ref{easy:estimate} to the map $f^m$ we obtain
	\[
		\|f^m_*(z)-z\|_{\fil,\Z}\geq \|M\rtimes_{f^m} S^1\|_\Z-  (n+1)\|z\|_1 \geq m \cdot \|N_1\|-  (n+1)\|z\|_1\ ,
	\]
	hence
	\[
		\frac{  \|f^m_*(z)-z\|_{\fil,\Z}}{m}\geq \|N_1\|- \frac{(n+1)\|z\|_1}{m}\ .
	\]
	By taking the limit of this inequality as $m\to+\infty$ we conclude that
	\[
		\LV_\Z(f)\geq \|N_1\|>0\ ,
	\]
	as desired.
\end{proof}

In order to conclude the proof of Theorem~\ref{mainZ:thm} we now need to deal with the $2$-dimensional case, i.e.~we need to exhibit a closed orientable surface $\Sigma$
and a homeomorphism $f\colon \Sigma\to \Sigma$ such that $\LV_\Z(f)>0$, while $\|\Sigma\rtimes_f S^1\|_\Z^\infty=0$. To this aim, it suffices to set $\Sigma=T=S^1\times S^1$,
and let $f$ be any orientation preserving Anosov  self-homeomorphism of $T$. Indeed, in this case the mapping torus $T\rtimes_f S^1$ is a closed $3$-manifold supporting
a geometric structure modeled on Sol, hence  $\|T\rtimes_f S^1\|_\Z^\infty=0$ (see~\cite[Theorem~1]{fauser:stable}). Moreover, $\LV_\Z(f)>0$ thanks to the following:

\begin{prop}\label{Anosov:prop}
	Let $f \colon T \to T$ be an orientation preserving  Anosov self-homeomorphism of the torus. Then $\LV_\Z(f)>0$.
\end{prop}

\begin{proof}
	Let $E_m=T \rtimes_{f^{m}} S^1$ denote the mapping torus relative to the homeomorphism $f^m\colon T\to T$. Since $f$ is an Anosov map, $f_*:H_1(T,\Z) \to H_1(T,\Z)$ is a matrix in $\mathrm{SL}_2(\Z)$ with eigenvalues $\lambda$ and $1/\lambda$, where $|\lambda|>1$; up to replacing $f^2$ with $f$, we can suppose $\lambda>1$. Thanks to ~\cite[Lemma~10]{sakuma}, we get
	\begin{equation}\label{tors}
		|\operatorname{tors} (H_1(E_m,\Z))| = \operatorname{tr} (f_*^m) -2 = \lambda^m +\lambda^{- m}-2 \ .
	\end{equation}
	If $c \in Z_3(E_m,\Z)$ is any integral fundamental cycle and $k$ is the number of $3$-simplices appearing in the reduced form of $c$, then ~\cite[Theorem~3.2]{sauer_volumegrowth} gives
	\[\|c \|_1 \geq k \geq (6\log(4))^{-1}{\log|\operatorname{tors} H_1(E_m,\Z)|} \ , \]
	hence
	\[
		\|E_m\|_\Z\geq (6\log(4))^{-1}{\log|\lambda^m +\lambda^{- m}-2|}\ .
	\]
	Now, if $z$ is an integral fundamental cycle of $T$, then Lemma \ref{easy:estimate} gives
	\begin{align*}
		\frac{\|f_*^m(z)- z\|_{\fil,\Z}}{m}\geq &
		\frac{\|E_m \|_\Z}{m}-\frac{3\|z \|_1}{m}                                                                                    \\
		\geq                                    & \frac{(6\log(4))^{-1}|\log(\lambda^m +\lambda^{- m}-2)|}{m}-\frac{3\|z \|_1}{m}\ .
	\end{align*}
	The result follows by taking the limit as $m\to \infty$.
\end{proof}

As stated in the introduction, Theorem~\ref{mainZ:thm} implies that for every $n \geq 2$ there exist an $n$-manifold $M$ and an element $[f]\in\mcg(M)$ such that $\LV_\R(f)=0$, while $\LV_\Z(f)\neq 0$.
As a consequence, if $z$ is an integral fundamental cycle for $M$ and $c_m=f^m_*(z)-z$, then
\[
	\lim_{m\to +\infty} \frac{ \|c_m\|_{\fil,\R}}{\|c_m\|_{\fil,\Z}}=\lim_{m\to +\infty} \frac{ \|c_m\|_{\fil,\R}/m}{\|c_m\|_{\fil,\Z}/m}=
	\frac{\LV_\R(f)}{\LV_\Z(f)}=0\ .
\]
This proves Corollary~\ref{filling:cor} from the introduction.

\bibliography{biblio_omeo}
\bibliographystyle{alpha}
\end{document}